\documentclass{amsart}

\AtBeginDocument{%
  \providecommand\BibTeX{{%
    \normalfont B\kern-0.5em{\scshape i\kern-0.25em b}\kern-0.8em\TeX}}}

\usepackage{xcolor}

\usepackage{amsfonts, amsmath, amssymb, amsthm, mathtools, hyperref}

\newtheorem{theorem}{Theorem}[section]

\newtheorem{lemma}[theorem]{Lemma}

\newtheorem{proposition}[theorem]{Proposition}
\newtheorem{corollary}[theorem]{Corollary}
\theoremstyle{definition}
 
\newtheorem{definition}[theorem]{Definition}

\newtheorem{remark}[theorem]{Remark}

\newcommand{\mbk}{\mathbb K}

\newcommand{\vp}{\Phi}
\newcommand{\del}{\partial}

\newcommand{\stl}{\left\{}
\newcommand{\str}{\right\}}

\newcommand{\A}{\mathcal{A}}
\newcommand{\om}{\omega}
\newcommand{\bom}{\bar{\omega}}
\newcommand{\e}{\eta}
\newcommand{\be}{\bar{\eta}}

\newcommand{\m}{\mathbf{m}}
\newcommand{\Q}{\mathcal{Q}}
\newcommand{\ap}{R}

\begin{document}

\title{The apolar algebra of a product of linear forms}
\author[M. DiPasquale, Z. Flores, C. Peterson]{Michael DiPasquale, Zachary Flores, Chris Peterson}

\address{Colorado State University\\
	Department of mathematics, Fort Collins, CO 80523}
\email{michael.dipasquale@colostate.edu}

\email{flores@math.colostate.edu}

\email{peterson@math.colostate.edu}

\begin{abstract}
Apolarity is an important tool in commutative algebra and algebraic geometry which studies a form, $f$, by the action of polynomial differential operators on $f$.  The quotient of all polynomial differential operators by those which annihilate $f$ is called the \textit{apolar algebra} of $f$.  In general, the apolar algebra of a form is useful for determining its Waring rank, which can be seen as the problem of decomposing the supersymmetric tensor, associated to the form, minimally as a sum of rank one supersymmetric tensors. In this article we study the apolar algebra of a product of linear forms, which generalizes the case of monomials and connects to the geometry of hyperplane arrangements.  In the first part of the article we provide a bound on the Waring rank of a product of linear forms under certain genericity assumptions; for this we use the defining equations of so-called star configurations due to Geramita, Harbourne, and Migliore.  In the second part of the article we use the computer algebra system \textsc{Bertini}, which operates by homotopy continuation methods, to solve certain rank equations for catalecticant matrices.  Our computations suggest that, up to a change of variables, there are exactly six homogeneous polynomials of degree six in three variables which factor completely as a product of linear forms defining an irreducible multi-arrangement and whose apolar algebras have dimension six in degree three.  As a consequence of these calculations, we  find six cases of such forms with cactus rank six, five of which also have Waring rank six.  Among these are products defining subarrangements of the braid and Hessian arrangements.
\end{abstract}



\thanks{
The third author was partially supported by NSF ATD-1712788 and NSF CCF-1830676.
}

\maketitle

\section{Introduction}
Given a homogeneous polynomial $f$ of degree $d$, the \textit{apolar algebra} $R_f$ is the ring of polynomial differential operators modulo those which annihilate $f$.  This algebra has been studied for a variety of reasons; in particular the apolar algebra of a form of degree $d$ is always an Artinian Gorenstein algebra with socle degree $d$ and every Artinian Gorenstein algebra with socle degree $d$ can be represented as the apolar algebra of a form of degree $d$.  This explicit correspondence, via the apolar algebra, between forms of degree $d$ and  Artinian Gorenstein algebras with socle degree $d$ is well exposited by Iarrabino and Kanev in~\cite{ps}.  The apolar algebra of a homogeneous polynomial $f$ of degree $d$ is also key to studying the \textit{Waring rank} of $f$ -- this is the smallest integer $r$ for which there exist linear forms $\ell_1,\ldots,\ell_r$ so that $f=\ell_1^d+\cdots+\ell_r^d$ (we call such a representation a \textit{Waring decomposition}).  The Waring rank often depends on the field chosen -- in this note we will work over an algebraically closed field.  Note that homogeneous polynomials of degree $d$ correspond to supersymmetric $d$-dimensional tensors and that the $d^{th}$ power of a linear form corresponds to a rank 1 supersymmetric $d$-dimensional tensor. Through this correspondence, Waring rank connects to tensor rank and Waring decomposition to tensor decomposition.

In this note we study the apolar algebra of a form $f$ of degree $d$ which can be written as a product of $d$, not necessarily distinct, linear forms.  Such forms correspond geometrically to hyperplane arrangements (in the case of distinct  linear forms) and hyperplane multi-arrangements (in the case of non-distinct  linear forms).  To simplify exposition, we conflate a multi-arrangement with its defining equation.  For instance, if we refer to the Waring rank of a multi-arrangement, we mean the Waring rank of its defining equation.  Our inspiration for studying this problem stems largely from the thesis of Max Wakefield~\cite{W06}, where several questions are posed about apolar algebras of multi-arrangements.  In particular, we study when the apolar algebra of a multi-arrangement is a complete intersection.

If the apolar algebra of a form is a complete intersection, it is often easier to compute its Waring rank.  Two important classes of examples (all multi-arrangements) serve to illustrate this point.  The first is the case of a monomial, whose apolar algebra is generated by powers of variables.  The Waring rank of monomials over the field of complex numbers is completely determined in~\cite{carlini2012solution}.  The second class is when $f$ is the fundamental skew invariant of a complex reflection group $W$, which is the product of the linear forms defining the pseudo-reflections of $W$.  In this case the apolar algebra $R_f$ is isomorphic to the ring of covariants of $W$~\cite[Chapter~26]{K01}, which is the quotient of the polynomial ring by the ideal generated by invariants of $W$.  This is a complete intersection since the ring of invariants is itself a polynomial ring by the celebrated Chevalley-Shephard-Todd theorem.  In~\cite{TW15}, Teitler and Woo determine the Waring rank of (and a Waring decomposition of) the fundamental skew invariant of a complex reflection arrangement under some mild conditions.

Following a section providing preliminary background material, we briefly discuss \textit{reducible} arrangements, which are arrangements that can be written as a product of lower dimensional arrangements.  In Section 4 we make use of the defining equations of \textit{star configurations} determined by Geramita, Harbourne, and Migliore~\cite{GHG13} to give a lower bound on the initial degree of the apolar algebra of a generic arrangement (Proposition~\ref{prop:genericlowdeg}).  We give two corollaries to Proposition~\ref{prop:genericlowdeg} -- the first is a lower bound on the size of a generic arrangement whose apolar ideal is a complete intersection and the second is a lower bound on the Waring rank of a generic arrangement.  Section 5 gives a case study of six lines in $\mathbb P^2$. In particular, we use the numerical computer  algebra system \textsc{Bertini}~\cite{BHSW06} to compute what we suspect is a complete list of irreducible multi-arrangements consisting of six lines (counting multiplicity) and annihilated by at least three cubics.  We record this list in Theorem*~\ref{thm:sixlines} (the star indicates this is a computational result that needs further verification).  This leads to what we expect is a complete list of irreducible multi-arrangements consisting of six lines that have cactus rank equal to six (all but one of these also have Waring rank equal to six).  \textsc{Macaulay2}~\cite{m2},~\textsc{Sage}~\cite{sagemath}, and \textsc{Bertini} scripts we used to find this list and check the resulting Waring ranks can be found under the Research tab at \url{https://midipasq.github.io/}.  The final section of the paper provides closing comments and gives suggestions  for  further research.

\section{Preliminaries}
Let $\mbk$ be an algebraically closed field of characteristic zero and put $R = \mbk[X_0,\ldots, X_n]$.
Let $S=\mbk[x_0,\ldots, x_n]$ be the $R$-module defined by $R$ acting on $S$ via partial differentiation.  That is, if $f\in S$ and $\vp\in R$,
\[
\vp\circ f = \vp\left(\frac{\del}{\del x_0},\ldots, \frac{\del}{\del x_n}\right)f.
\]
This is known as the \textit{apolar} action of $R$ on $S$.  The expository article of Geramita~\cite{G95} is an excellent introduction to applications of apolarity, the book of Iarrabino and Kanev~\cite{ps} can be used to go into more detail, and the article of De Paris~\cite{de2018seeking} gives a recent summary of apolarity and tensor rank.


Given a form $f\in S$, the \textit{apolar ideal} of $f$ is
\[
\text{Ann}_R(f) = \stl \vp\in R :\vp\circ f = 0 \str.
\]
We write $\ap_f=R/\text{Ann}_R(f)$; this is the \textit{apolar algebra} of $f$.  The apolar algebra $\ap_f$ is a graded Artinian Gorenstein algebra, and every graded Artinian Gorenstein algebra arises in this way~\cite[Lemma~2.12]{ps}.

Now suppose $f\in S_d$ (where $S_d$ denotes the degree $d$ forms in $S$).  A \textit{Waring decomposition} of $f$ is a decomposition $f=c_1\ell_1^d+\cdots+c_k\ell_k^d$, where $\ell_1,\ldots,\ell_k$ are linear forms and $c_1,\ldots,c_k\in\mbk$ (we do not strictly need $c_1,\ldots,c_k$ since $\mbk$ is algebraically closed, but it will be useful for us to consider these).  The smallest number of linear forms needed in a Waring decomposition of $f$ is the \textit{Waring rank} of $f$.  The following lemma relates the apolarity action and Waring decompositions.  See~\cite[Lemma~1.15]{ps} for a proof.  In what follows, we say a linear form $\ell=\sum_{i=0}^n a_ix_i\in\mbk[x_0,\ldots,x_n]$ is \textit{dual} to the point $P=[a_0: \dots : a_n]\in\mathbb{P}_\mbk^n$. Any non-zero constant multiple of $\ell$ is of course dual to the same point $P$.

\begin{lemma}[Apolarity Lemma]\label{lem:Apolarity}
	Let $f\in S=\mbk[x_0,\ldots,x_n]$ be a form of degree $d$, $X=\{P_1,\ldots,P_k\}\subset \mathbb{P}_\mbk^n$ a set of points, and $I_X\subset R$ its corresponding ideal.  Write $\ell_1,\ldots,\ell_k$ for linear forms in $S$ dual to the points $P_1,\ldots,P_k$.  Then $f=c_1\ell_1^d+\ldots+c_k\ell_k^d$ for some constants $c_1,\ldots,c_k$ if and only if $I_X\subset \text{Ann}_R(f)$.
\end{lemma}

From the apolarity lemma we see that the Waring rank of a form is the same as the minimum degree of a zero-dimensional radical ideal contained in its apolar ideal.  A related notion is the \textit{cactus rank} of a form; this is the minimum degree of a zero-dimensional saturated ideal contained in its apolar ideal (we will see this notion in Section~\ref{sec:6Lines}).

We will focus on forms $f\in S=\mbk[x_0,\ldots,x_n]$ which decompose as a product of (not necessarily distinct) linear forms as $f=\ell_1^{m_1}\cdots\ell_k^{m_k}$.  If $g\in S$, write $V(g)$ for the set of points in $\mbk^{n+1}$ at which $g$ vanishes.  A natural geometric object to attach to the product $f=\ell_1^{m_1}\cdots\ell_k^{m_k}$ is the \textit{multi-arrangement} $(\A,\m)$ where $\A=\cup_{i=1}^k V(\ell_i)$ is the union of the hyperplanes $V(\ell_i)\subset\mbk^{n+1}$ and $\m$ is a function which assigns to each hyperplane $H\in\A$ the integer $\m(H)$, where $\m(H)$ is the power to which the corresponding linear form appears in $f$.  We put $|\m|=\sum_H \m(H)$, which is the degree of the polynomial $f$.  If $\m(H)=1$ for all $H\in\A$ we will say $(\A,\m)$ is a \textit{simple} arrangement and write $\A$ instead of $(\A,\m)$.  Given a multi-arrangement $(\A,\m)$ we define $\Q(A,\m):=\prod_{H\in\A} \alpha_H^{\m(H)}$, where $\alpha_H$ is a choice of linear form vanishing on $H$.  If $\A$ is simple then we write $\Q(\A)$ for the product $\prod_{H\in\A} \alpha_H$.  We call $\Q(\A,\m)$ and $\Q(\A)$ the \textit{defining polynomial} of the multi-arrangement and arrangement, respectively.  Moreover we write $|\A|$ for the number of hyperplanes in $\A$, so that if $f=\Q(\A,\m)$, then $|\A|$ is the number of distinct linear factors of $f$.  For simplicity, throughout this note we will conflate a multi-arrangement or arrangement with its defining polynomial.  For instance, by ``the apolar algebra of an arrangement" we will mean the apolar algebra of its defining equation.

If $\A_1=\cup_{i=1}^s G_i\subset V\cong \mbk^{n}$ and $\A_2=\cup_{j=1}^t H_j  \subset W\cong\mbk^{m}$ are two simple arrangements, then the \textit{product} of $\A_1$ and $\A_2$ is defined by
\[
\A_1\times \A_2=\left(\cup_{i=1}^s G_i\times W \right)\cup \left(V \times \cup_{j=1}^t H_j \right)\subset V\times W
\]
If $(\A_1,\m_1)$ and $(\A_2,\m_2)$ are multi-arrangements, the product multi-arrangement $(\A_1\times\A_2,\m)$ satisfies $\m(H\times W)=\m(H)$ if $H\in\A_1$ and $\m(V\times G)=\m(G)$ if $G\in\A_2$.  Following~\cite{OT92}, we will say that a simple arrangement $\A$ is \textit{reducible} if, after a change of coordinates, $\A=\A_1\times\A_2$ for some simple arrangements $\A_1$ and $\A_2$.  Otherwise we say that $\A$ is \textit{irreducible}.

Suppose $\A\subset\mbk^n$ is a reducible arrangement and $\Q(\A)$ is its defining polynomial.  Then there is a change of variables so that $\A=\A_1\times\A_2$, where $\A_1\subset\mbk^s$ and $\A_2\subset\mbk^t$ for some positive integers $s,t$ satisfying $s+t=n$.  Put $S_1=\mbk[x_1,\ldots,x_s]$ and $S_2=\mbk[y_1,\ldots,y_t]$.  Then, under this change of variables, $\Q(\A)=\Q(\A_1)\Q(\A_2)$.  Algebraically, the defining polynomials of reducible arrangements are those which, after an appropriate change of variables, split as a product of two defining polynomials in disjoint sets of variables.

In this note we only consider hyperplane arrangements all of whose hyperplanes pass through the origin (these are called \textit{central} arrangements).  Hence we will freely pass between a central arrangement in $\mbk^{n+1}$ and its natural quotient in $\mathbb{P}^n$ -- this does not affect the algebra.


\section{Products of one and two dimensional arrangements}
In this section we observe that if $(\A,\m)$ is reducible, so $(\A,\m)=(\A_1,\m_1)\times(\A_2,\m_2)$ after a change of variables, then $\ap_{f}\cong \ap_{f_1}\otimes_{\mbk}\ap_{f_2}$, where $f=\Q(\A,\m),f_1=\Q(\A_1,\m_1),$ and $f_2=\Q(\A_2,\m_2)$.  Our observation hinges on the following proposition.  We suspect this is well-known but we include a proof since we were not able to find one in the literature.

\begin{proposition}\label{prop:disjointvariables}
	Suppose $s$ and $t$ are positive integers, $f\in S_1=\mbk[x_1,\ldots,x_s]$ and $g\in S_2=\mbk[y_1,\ldots,y_t]$.  Put $S=S_1\otimes_{\mbk} S_2$.  Viewing $S$ as the polynomial ring $\mbk[x_1,\ldots,x_s,$ $y_1,\ldots,y_t]$, we abuse notation by writing $fg$ for the simple tensor $f\otimes g\in S$.  We write $R_1,R_2,$ and $R$ for the polynomial rings dual to $S_1,S_2,$ and $S$.  Then
	\begin{enumerate}
		\item $\ap_{fg}\cong (\ap_1)_f\otimes_{\mbk} (\ap_2)_g$ and
		\item $\text{Ann}_R(fg)= \text{Ann}_{R_1}(f)R_2+\text{Ann}_{R_2}(g)R_1$
	\end{enumerate}
\end{proposition}
\begin{proof}
	Since $\text{Ann}_{R_1}(f)R_2+\text{Ann}_{R_2}(g)R_1$ is the kernel of the natural map from $R$ to $\ap_f\otimes \ap_g$, it is clear that (1) and (2) are equivalent.  We prove (2).
	
	Suppose that $\vp=\sum_{\alpha,\beta} c_{\alpha,\beta} X^\alpha Y^\beta \in R$, where $\alpha=(\alpha_0,\cdots,\alpha_s),\beta=(\beta_0,\ldots,\beta_t)\in \mathbb{Z}_{\ge 0}^{t+1}$, $X^\alpha=X_0^{\alpha_0}\cdots X_s^{\alpha_s}$, $Y^\beta=Y_0^{\beta_0}\cdots Y_t^{\beta_t}$, and $c_{\alpha,\beta}\in\mbk$.
	Then
	\[
	\vp\circ (fg)=\sum_{\alpha,\beta} c_{\alpha,\beta} \frac{\partial f}{\partial x^\alpha} \frac{\partial g}{\partial y^\beta}.
	\]
	Similarly, if $\varphi_1\in R_1$ and $\varphi_2\in R_2$, then $\varphi_1\varphi_2\circ fg=(\varphi_1\circ f)(\varphi_2\circ g)$.  From this observation it is clear that $\text{Ann}_{R_1}(f)R_2+\text{Ann}_{R_2}(g)R_1\subseteq \text{Ann}_R(fg)$.
	
	We prove that $\text{Ann}_{R}(fg)\subseteq \text{Ann}_{R_1}(f)R_2+\text{Ann}_{R_2}(g)R_1$.  For this we consider several maps: $\alpha_f:R_1\to S_1$ given by 
	$\varphi\to \varphi\circ f$, $\alpha_g:R_2\to S_2$ by $\varphi\to \varphi\circ g$, the tensor product maps $\alpha'_f:=\alpha_f\otimes_\mbk \mbox{id}_{R_2}:R_1\otimes_\mbk R_2\to S_1\otimes_\mbk R_2$ and $\alpha'_g := \mbox{id}_{S_1}\otimes_\mbk \alpha_g:S_1\otimes_\mbk R_2\to S_1\otimes_\mbk S_2$.  By the above observations, $\text{Ann}_R(fg)=\ker(\alpha'_g\circ\alpha'_f)$.
	
	Suppose $\vp=\sum_{\alpha,\beta} c_{\alpha,\beta} X^\alpha Y^\beta \in \text{Ann}_R(fg)$.  Then
	\begin{equation}\label{eq:zero}
	\vp\circ fg=\sum_{\alpha,\beta} c_{\alpha,\beta} \frac{\partial f}{\partial x^\alpha} \frac{\partial g}{\partial y^\beta}=0.
	\end{equation}
	Suppose the monomial $x^\gamma$ appears in $\frac{\partial f}{\partial x^\alpha}$ with coefficient $d_{\gamma,\alpha}\in\mbk$.  Equating coefficients of $x^\gamma$ in Equation~\eqref{eq:zero} yields
	\[
	x^\gamma\sum_{\alpha,\beta} d_{\gamma,\alpha}c_{\alpha,\beta} \frac{\partial g}{\partial y^\beta}=0.
	\]
	It follows that $\sum_{\alpha,\beta} d_{\gamma,\alpha} c_{\alpha,\beta} Y^\beta\in\ker(\alpha'_g)=\text{Ann}_{R_2}(g)$.  Thus
	\[
	\alpha'_f(\vp)=\sum_{\alpha,\beta} c_{\alpha,\beta} \frac{\partial f}{\partial x^\alpha} Y^\beta\in \text{Ann}_{R_2}(g)\alpha_f(R_1).
	\]
	Notice that
	\[
	\alpha'_f(\text{Ann}_{R_1}(f)R_2+\text{Ann}_{R_2}(g)R_1)=\text{Ann}_{R_2}(g)\alpha_f(R_1).
	\]
	Since $\alpha'_f(\text{Ann}_R(fg))\subseteq \text{Ann}_{R_2}(g)\alpha_f(R_1)$ and $\text{ker}(\alpha'_f) = \text{Ann}_{R_1}(f)R_2$, we have $\text{Ann}_R(fg)\subseteq \text{Ann}_{R_1}(f)R_2+\text{Ann}_{R_2}(g)R_1$, as desired.
\end{proof}

\begin{corollary}\label{cor:productMatroid}
	Suppose $S\cong S_1\otimes_\mbk \cdots \otimes_\mbk S_k$, where $S_i$ is a polynomial ring in one or two variables for $i=1,\ldots,k$.  If a form $f\in S$ factors as $f=f_1\cdots f_k$ where $f_i\in S_i$ for $i=1,\ldots,k$, then $\text{Ann}_R(f)$ is a complete intersection.
\end{corollary}
\begin{proof}
	It is well known that the apolar algebra of a homogeneous polynomial in one or two variables is a complete intersection (since Gorenstein coincides with complete intersection in one and two variables).  The corollary follows directly from this fact and Proposition~\ref{prop:disjointvariables}.
\end{proof}

\begin{remark}
	Over an algebraically closed field it is clear that the factors $f_1,\ldots,f_k$ in Corollary~\ref{cor:productMatroid} are in fact products of linear forms.
\end{remark}

\begin{remark}
	Corollary~\ref{cor:productMatroid} shows that the apolar algebra of a multi-arrangement which is a product of one and two dimensional arrangements is a complete intersection.  One may ask the reverse question: if the apolar algebra of $\Q(\A,\m)$ is a complete intersection for every choice of multiplicity $\m$, is $\A$ necessarily a product of one and two dimensional arrangements?  A similar question has an affirmative answer: in~\cite{ATY09} it is proved that if the module of multi-derivations $D(\A,\m)$ is free for every multiplicity $\m$, then $\A$ is indeed a product of one and two dimensional arrangements.
\end{remark}

\section{Generic arrangements}

In this section we derive a lower bound on the initial degree of the apolar ideal of a \textit{generic} arrangement $\A\subset\mathbb{P}^n$ with at least $n+1$ hyperplanes (Proposition~\ref{prop:genericlowdeg}).  All arrangements in this section are simple arrangements.

\begin{definition}\label{def:generic}
An arrangement in $\mathbb{P}^n$ is generic if the intersection of any $k$ of its hyperplanes has codimension $\min\{k,n+1\}$.	
\end{definition}

In preparation we give several lemmas and definitions.  Given a form $G\in R$, the \textit{gradient of $G$} is the vector $\nabla G := \left(\frac{\partial G}{\partial X_0},\ldots, \frac{\partial G}{\partial X_n}\right)$.

\begin{lemma}\label{lem3} 
	Suppose $g\in S$ is a homogeneous polynomial and write $f=\ell g$ for some linear form $\ell$.  Let $F\in R$ be homogeneous of degree $d\geq 1$.  Then, if we abuse notation and write $\ell$ for the corresponding linear form in $R$, we have 
	\[
	F\circ f = (\nabla F\cdot \nabla \ell)\circ g + \ell\left(F\circ g\right).
	\]
	(Here $\nabla F\cdot \nabla \ell$ denotes the \textit{dot product}.)  In particular, if $f=\ell_1\ell_2\cdots\ell_t$ is a product of $t\ge n$ linear forms, $n$ of which are linearly independent, then there is an $\ell\in \stl \ell_1,\ldots,\ell_t\str$ such that $\nabla F\cdot \nabla \ell$ is nonzero.
\end{lemma}

\begin{proof} 
	
	Write $\ell = a_0x_0 + \cdots + a_nx_n$.  First, let $F$  be a monomial of degree $d$, say $F = X_{i_1}^{d_1}\cdots X_{i_t}^{d_t}$, where $d_1,\ldots, d_t$ are positive.  Then it is easy to see that $F\circ f$ is given by 
	
	\begin{multline*}\left(\sum_{j=1}^t d_ja_jX_{i_1}^{d_1}\cdots X_{i_j}^{d_j-1}\cdots X_{i_t}^{d_t}\right)\circ g + \ell(F\circ g)=
	\\
	(\nabla F\cdot \nabla \ell)\circ g+\ell(F\circ g)\tag{$\star$}
	\end{multline*}
	
	By linearity of the gradient, $(\star)$ holds for arbitrary polynomials $F$.  The rest is clear.
\end{proof}


\begin{definition}
	If $f$ is a form, the $k$\textit{th order Jacobian} of $f$ is the ideal generated by all partials of $f$ of order $k$ and is denoted by $J^k(f)$.
\end{definition}

\begin{remark}\label{rem:sing}
	The Jacobian of $f$ is $J^1(f)$; geometrically, $V(J^1(f))$ is the singular locus of $f$.  Analogously, $V(J^k(f))$ is the set of singular points with multiplicity at least $k+1$.
\end{remark}

\begin{remark}
	Since we assume $f$ is homogeneous, the Euler identity $\sum x_i\frac{d g}{d x_i}=\deg(g)\cdot g$ applied repeatedly to $f$ and its partials yields the containments $(f) \subset J^1(f)\subset J^2(f)\subset\cdots\subset J^k(f)$.  Geometrically, this yields a nested sequence of subvarieties of the hypersurface $V(f)$ ordered according to the severity of the singularities.
\end{remark}

\begin{remark}
	If $f$ is a form of degree $d$, the degree $k$ component of the apolar algebra $(\ap_f)_k$, is isomorphic (as a vector space over $\mathbb K$) to $J^{d-k}(f)_k$ via apolarity.  Hence $\mbox{Ann}_R(f)_k=0$ if and only if $J^{d-k}(f)$ is the $k$th power of the maximal ideal.
\end{remark}

According to Remark~\ref{rem:sing}, if $f$ is a product of linear forms, then $V(J^k(f))$ is exactly those points which lie at the intersection of at least $k+1$ of the hyperplanes defined by the linear forms whose product is $f$.  Now we arrive at the crucial point: if $f=\Q(\A)$ for a \textit{generic} arrangement, $V(J^k(f))$ is precisely the union of all codimension $k+1$ intersections of hyperplanes from $\A$.  Thus $V(J^k(f))$ is a \textit{star configuration}~\cite{GHG13}; a star configuration is by definition the union of all codimension $c$ intersections of a generic arrangement (in~\cite[Definition~2.1]{GHG13} the property of \textit{meeting properly} is exactly what we mean by a \textit{generic} arrangement).  In~\cite{GHG13} it is shown that the ideal of codimension $c$ intersections of an arrangement of $|\A|$ hyperplanes is generated by all distinct products of $|\A|-c+1$ of the linear forms defining $\A$.

\begin{lemma}\label{lem:Colon}
	Suppose $f$ decomposes non-trivially as a product $f=gh$; write $I=\mbox{Ann}_R(h)$ and $I'=\mbox{Ann}_R(f)=\mbox{Ann}_R(gh)$.  If $D\in I'_{k}\setminus I_{k}$, then $g\in J^{k-1}(h):(D\circ h)$.
\end{lemma}
\begin{proof}
	Repeatedly using the product rule yields that $D\circ gh=g(D\circ h)+T$, where $T\in J^{k-1}(h)$.  Since $D\circ gh = 0$, this gives the result.  
\end{proof}

\begin{corollary}\label{cor:PropagationGenericDegrees}
	Suppose $f$ is a product of at least $n+2$ distinct linear forms defining a generic arrangement $\mathcal{A}$ in $\mathbb{P}^n$.  Factor $f$ as a product $f=gh$ so that $\deg(h)\geq n+1$.  Write $I=\mbox{Ann}_R(h)$ and $I'=\mbox{Ann}_R(f)=\mbox{Ann}_R(gh)$.  If $I_{k}=0$ for any $k\le n$ then $I'_k=0$.
\end{corollary}
\begin{proof}
	
	Suppose to the contrary that $D\in I'_k$ and $D\neq 0$.  By Lemma~\ref{lem:Colon}, $g\in J^{k-1}(h):(D\circ h)$.  Write $h=\ell_1\ell_2\cdots\ell_t$, where $t\geq n+1$; then $V(J^{k-1}(h))$ is the union of linear spaces which are the intersections of at least $k$ of the hyperplanes $V(\ell_1),\cdots,V(\ell_t)$.  This is nonempty since $k\le n<t$.  As $\mathcal{A}$ is a generic arrangement, none of the factors of $g$ vanish along any component of $V(J^{k-1}(h))$; in other words $g$ is not in any prime ideal that comprises the intersection that is the radical of $J^{k-1}(h)$.  This means that $g\in J^{k-1}(h):(D\circ h)$ only if $D\circ h$ is in every minimal prime of $J^{k-1}(h)$.  In other words, $D\circ h$ is in the radical of $J^{k-1}(h)$.  Let $K=\sqrt{J^{k-1}(h)}$; this is the ideal of the union of linear spaces which are the intersections of $k$ of the hyperplanes $V(\ell_1),\cdots,V(\ell_t)$.  As previously noted, this is a star configuration, and by ~\cite[Proposition~2.9]{GHG13}, $K$ is generated by all possible products of $t-k+1$ of the linear forms $\ell_1,\ldots,\ell_t$.  On the other hand $D\circ h$ has degree $t-k$, so $D\circ h\notin K$.  With this contradiction, we must have $I'_k=0$.
\end{proof}

\begin{remark}
	Consider the $A_3$ arrangement in $\mathbb{P}^2$, defined by $f=xyz(x-y)(x-z)(y-z)$.  Write $f=gh$ with $g=y-z$ and $h=xyz(x-y)(x-z)$.  Set $I'=\mbox{Ann}_R(f)$ and $I=\mbox{Ann}_R(h)$.  Then $I_2=0$ but $I'_2\neq 0$.  Thus the hypothesis that $\mathcal{A}$ is generic in Corollary~\ref{cor:PropagationGenericDegrees} is necessary.
\end{remark}

Now we give the main result of this section -- a bound on the initial degree of the apolar ideal of a generic arrangement.  For an ideal $I\subset R$ we will denote by $\alpha(I)$ its initial degree, that is, the smallest degree $d$ for which $I_d\neq 0$.


\begin{proposition}\label{prop:genericlowdeg}
	Suppose $\A$ is a generic arrangement of at least $n+1$ hyperplanes in $\mathbb{P}^n$ and $f=\Q(\A)$.  Then $\alpha(\mbox{Ann}_R(f))\ge \min\{|\A|-n+1,n+1\}$.
\end{proposition}
\begin{proof}
	We first prove by induction on $|\A|$ that if $n+1\le|\A|\le 2n$, then $\alpha(\mbox{Ann}_R(f))\ge |\A|-n+1$.  If $|\A|=n+1$ then without loss of generality, $f=x_0x_1\cdots x_n$ and $\mbox{Ann}_R(f)=(x_0^2,\ldots,x_n^2)$, so $\alpha(\mbox{Ann}_R(f))=2=|\A|-n+1$.
	
	Suppose now that $n+1 <|\A| \leq 2n$, and additionally suppose for a contradiction that there is some $D\in \mbox{Ann}_R(f)_{|\A|-n}$.  Since $\A$ is defined by more than $n$ linearly independent linear forms, by Lemma~\ref{lem3} there is some $\ell\in\mathcal A$ so that $\nabla\ell\cdot\nabla D\neq 0$.  Writing $f=g\ell$, with $\deg(g) = n$, and using Lemma~\ref{lem3} again, we have 
	\[
	0 = D\circ f=(\nabla\ell\cdot \nabla D)\circ g+\ell(D\circ g).
	\]
	Suppose $D\circ g=0$, so that $(\nabla\ell\cdot \nabla D)\circ g=0$.  Now $\deg(\nabla\ell\cdot\nabla D)=|\A|-n-1$, and by induction $\alpha(\mbox{Ann}_R(g))\ge |\A|-1-n+1=|\A|-n$.  With this contradiction, $D\circ g \neq 0$.
	
	With the above, $\ell(D\circ g))=-(\nabla\ell\cdot \nabla D)\circ g$, so $\ell(D\circ g)\in J^{|\A|-n-1}(g)$.  Write $K = \sqrt{J^{|\A|-n-1}(g)}$, so that $K$ is the ideal defining all possible intersections of $|\A|-n$ hyperplanes of $g$; by~\cite{GHG13}, $\alpha(K)=(|\A|-1)-(|\A|-n)+1=n$.  Since $\deg(D\circ g)=(|\A|-1)-(|\A|-n)=n-1$, $D\circ g\notin K$.  Since $K$ is radical, $\ell$ must be in at least one minimal prime of $K$.  This would imply that $V(\ell)$ passes through a codimension $|\A|-n$ intersection of $\A$.  As $|\A|\le 2n$, $K$ is not the homogeneous maximal ideal, so that this contradicts that $\A$ is a generic arrangement.  Hence no such $D$ can exist, and it follows that $\alpha(\mbox{Ann}_R(f))\ge |\A|-n+1$.
	
	If $|\A|\ge 2n$ we prove by induction on $|\A|$ that $\alpha(\mbox{Ann}_R(f))\ge n+1$.  The base case $|\A|=2n$ has already been shown.  If $|\A|>2n$ then the result follows from Corollary~\ref{cor:PropagationGenericDegrees}.
\end{proof}

\begin{corollary}\label{cor:genericci}
If $\A$ is a generic arrangement of at least $n+2$ hyperplanes in $\mathbb{P}^n$ whose apolar ideal is a complete intersection, then $|\A|\ge n(n+1)$.
\end{corollary}
\begin{proof}
	Put $f=\Q(\A)$.  If $\mbox{Ann}_R(f)$ is a complete intersection generated in degrees $d_0\le \ldots \le d_n$, then $(d_0-1)+(d_1-1)+\cdots+(d_n-1)=|\A|$, so $d_0+\cdots+d_n=|\A|+n+1$.  With this notation, $\alpha(\mbox{Ann}_R(f))=d_0$, and this gives $ d_0 \leq (|\A|+n+1)/(n+1)$.
	
	It is straightforward to check that if $n+1<|\A|\le 2n$ then the lower bound for $\alpha(\mbox{Ann}_R(f))$ from Proposition~\ref{prop:genericlowdeg} is strictly larger than $(|\A|+n+1)/(n+1)$, so $\mbox{Ann}_R(f)$ cannot be a complete intersection.
	
	If $|\A|>2n$ then we obtain from Proposition~\ref{prop:genericlowdeg} that $n+1\le (|\A|+n+1)/(n+1)$ or equivalently $n(n+1)\le |\A|$, proving the corollary.
\end{proof}


\begin{corollary}\label{cor:genericWaring}
	The Waring rank of a generic arrangement $\A\subset\mathbb{P}^n$ with at least $n+1$ hyperplanes is at least $\min\{\binom{|\A|}{n},\binom{2n}{n}\}$.
\end{corollary}
\begin{proof}
	Put $f=\Q(\A)$.  By Proposition~\ref{prop:genericlowdeg}, $\alpha(\mbox{Ann}_R(f))\geq \min\{|\A|-n+1,n+1\}$.  Suppose $f=\sum_{i=1}^k \ell^{|\A|}_i$, and let $X=\{P_i\}_{i=1}^k$ be the dual points in $\mathbb{P}^n$ found by stripping off the coordinates of the linear forms $\ell_i$.  By Lemma~\ref{lem:Apolarity}, $I_X\subset \mbox{Ann}_R(f)$.  For this to happen, $X$ must impose independent conditions on forms of degree $d=\alpha(\mbox{Ann}_R(f))-1$.  In other words, $X$ must consist of at least as many points as the dimension of the vector space $S_d$, where $S=k[x_0,\ldots,x_n]$.  Since $\dim S_d=\binom{n+d}{n}$, this gives the result.  
\end{proof}

\begin{remark}
	As Corollary~\ref{cor:genericWaring} does not account for the degree of $\Q(\A)$, we suspect that Corollary~\ref{cor:genericWaring} is not optimal.  However we will see in Section~\ref{sec:6Lines} that, even if $\A$ is generic, $\Q(\A)$ can be annihilated by many forms of unexpectedly low degree.
\end{remark}

\section{Six lines in $\mathbb{P}^2$}\label{sec:6Lines}
In this section we give a computational case study of irreducible multi-arrangements in $\mathbb{P}^2(\mathbb{C})$ with six lines, counting multiplicity.  Our motivation for this case study comes from~\cite[Example~III.3.2]{W06}, where Wakefield observes that the determinant of the catalecticant matrix (defined below) is not enough to show that the apolar algebra of a generic arrangement of six lines in $\mathbb{P}^2(\mathbb{C})$ is not a complete intersection.  As a consequence of our case study, we can say with reasonable certainty that there are indeed no generic arrangements of six lines in $\mathbb{P}^2$ whose apolar algebra is a complete intersection.  Another motivation for this case study is that, according to Corollary~\ref{cor:genericci}, a generic line arrangement must have at least six lines in order for its apolar algebra to have the possibility of being a complete intersection.

By Proposition~\ref{prop:genericlowdeg}, a generic arrangement $\A$ of six lines cannot be annihilated by any quadrics.  It follows that if the apolar ideal of $\A$ is a complete intersection then it must be generated by a regular sequence of three cubics.  In the process of looking for generic arrangements with this property, computations in the computer algebra systems \textsc{Bertini} and \textsc{Macaulay2} led us to the following (computational) result.  We have no theoretical justification for this and have not used software such as \textsc{alphaCertified} for \textsc{Bertini} to give a theoretical guarantee that the computations are correct -- hence we will denote it as a Theorem*.

\begin{theorem}[*]\label{thm:sixlines}
	Suppose that $(\A,\m)$ is an irreducible multi-arrangement in $\mathbb{P}^2(\mathbb{C})$ and $|\m|=6$.  Put $f=\Q(\A,\m)$.  Suppose that $f$ satisfies either:
	\begin{enumerate}
		\item $\dim \text{Ann}_R(f)_3 \ge 3$
		\item $f$ has cactus rank at most $7$
	\end{enumerate}
	Then, up to a change in coordinates, $f$ is one of the following six polynomials:
	\begin{itemize}
		\item $f_1=xyz(x+y+z)(x+\alpha y+\bar{\alpha}z)(x+\bar{\alpha}y+\alpha z)$
		\item $f_2=xyz(x+y+z)(x+\e y+\om z)(x+\be y+ \bom z)$
		\item $f_3=xyz(x+y+z)(x+\bom y+\om z)(x+\be y+\e z)$
		\item $f_4=xyz(x+y+z)(x+\om y+\bom z)(x+\e y+ \be z)$
		\item $f_5=xyz(x+y+z)(x+y)(y+z)$
		\item $f_6=x^3yz(x+y+z)$,
	\end{itemize}
	where $\alpha=\exp(\frac{2\pi i}{3}), \om=\exp(\frac{\pi i}{3}), \eta=\frac{1}{\sqrt{3}}\exp(\frac{\pi i}{6})$, and the bar denotes complex conjugation.
	In fact, $\dim \text{Ann}_R(f_i)_3=4$ and the cactus rank of $f_i$ is $6$ for $1\le i\le 6$.  If instead we require that $f$ has Waring rank six, then $f$ must be one of $f_1,f_2,f_3,f_4,$ or $f_5$.
\end{theorem}

Before we discuss the simplifications and further computations leading to this result, we make some remarks about the polynomials listed in Theorem~\ref{thm:sixlines}.
\begin{itemize}
	\item The forms $f_1,f_2,f_3,$ and $f_4$ each define generic arrangements.
	\item After changing coordinates, $f_5$ is the defining polynomial of the $A_3$ braid arrangement.
	\item The product $f_1$ is exactly half of the well-known Hessian arrangement (see~\cite[Example~6.30]{OT92}).  
	\item For each of $i=1,\ldots,6$, $\text{Ann}_R(f_i)$ has four cubics (these are all listed in Table~\ref{tbl:6LinesS}).  In particular, Proposition~\ref{prop:genericlowdeg} is tight for these.
	\item For each of $i=1,\ldots,6$, the ideal $J_i$ generated by the elements of degree at most $3$ in $\text{Ann}_R(f_i)$ is the ideal of a zero-dimensional scheme of degree six in $\mathbb{P}^2$.  Except for $i=6$, the ideal $J_i$ is the ideal of six reduced points in $\mathbb{P}^2$.  These points are listed in Table~\ref{tbl:Waring}.  Via the apolarity lemma (Lemma~\ref{lem:Apolarity}), the ideals $J_i$ ($1\le i\le 5$) yield an explicit Waring decomposition for $f_i$, listed in Table~\ref{tbl:Waring}.  In Table~\ref{tbl:Waring}, the point $p_i$ is dual to the form $\ell_i$.
	\item It is known that the Waring (and cactus) rank of a form $f$ is at least as large as $\dim (\ap_f)_k$ for any $k$; since $\dim (\ap_{f_i})_k$ is maximal when $k=3$ and $\dim (\ap_{f_i})_3=6$ for each of $i=1,\ldots,6$, the minimum value the Waring (respectively, cactus) rank can be is $6$.  Thus the Waring rank of $f_1,\ldots,f_5$ is six.  For $f_1,\ldots,f_4$, this is the lower bound predicted by Corollary~\ref{cor:genericWaring}.
\end{itemize}


\begin{table}
	
	$\alpha=\exp(\frac{2\pi i}{3}),\om=\exp(\frac{\pi i}{3}),\eta=\frac{1}{\sqrt{3}}\exp(\frac{\pi i}{6})$
	
	\
	
	\renewcommand{\arraystretch}{2}
	
	\begin{tabular}{c|c}
		& \parbox[t][][t]{5.6 cm}{Annihilating cubics of $f_i$} \\[5 pt]
		\hline
		$f_1$ & \parbox[t][][t]{9 cm}{$X^3-Y^3,X^3-Z^3,XY^2+YZ^2+ZX^2, \\ X^2Y+Y^2Z+Z^2X$} \\[12 pt]
		\hline
		$f_2$  & \parbox[t][][t]{9 cm}{$X^2Z-XZ^2,3Y^2Z-3YZ^2+Z^3, X^3-3X^2Y+3XY^2, X^2Y-3XY^2+3Y^3+2XYZ-XZ^2-2YZ^2+Z^3$}\\[12 pt]
		\hline
		$f_3$ & 
		\parbox[t][][t]{9 cm}{$-6\e XY^2+6\e Y^3+6\e XZ^2+3\e YZ^2-6\e Z^3+X^3+2XY^2-3Y^3+2XYZ
			-4XZ^2-3YZ^2+3Z^3, -3\e XY^2+3 \e Y^3+X^2Y-Y^3,
			3\e XZ^2-3\e Z^3+X^2Z-3XZ^2+2Z^3, 3\e YZ^2+Y^2Z-2YZ^2$}\\[35 pt]
		\hline
		$f_4$ & 
		\parbox[t][][t]{9 cm}{
			$-6\be XY^2+6\be Y^3+6\be XZ^2+3\be YZ^2-6\be Z^3+X^3+2XY^2-3Y^3+2XYZ
			-4XZ^2-3YZ^2+3Z^3, -3\be XY^2+3 \be Y^3+X^2Y-Y^3,
			3\be XZ^2-3\be Z^3+X^2Z-3XZ^2+2Z^3, 3\be YZ^2+Y^2Z-2YZ^2$}
		\\[35 pt]
		\hline
		$f_5$ & \parbox[t][][t]{9 cm}{$X^2-XY+Y^2-YZ+Z^2, Y^3-2Y^2Z+2YZ^2, Z^4$ (generators for ideal)} \\[15 pt]
		\hline
		$f_6$ & \parbox[t][][t]{9 cm}{$Z^3, Y^2Z-YZ^2, Y^3, XY^2-XYZ+XZ^2+2YZ^2$}\\[5 pt]
	\end{tabular}
	
	\
	
	\caption{Annihilating cubics of forms $f_i$ in Theorem~\ref{thm:sixlines}}\label{tbl:6LinesS}
\end{table}


\begin{table}
	$\alpha=\exp(\frac{2\pi i}{3}),\beta=1+i,\om=\exp(\frac{\pi i}{3}),\eta=\frac{1}{\sqrt{3}}\exp(\frac{\pi i}{6})$
	
	\renewcommand{\arraystretch}{1.5}
	\setlength{\tabcolsep}{2pt}
	\begin{tabular}{c|c|c}
		Form & Dual Points & Waring Decomposition\\
		\hline
		$f_1$ & \parbox[t][][t]{2.3 cm}{$p_1=[\alpha:1:1]\\ p_2=[\bar{\alpha}:1:1]\\ p_3=[1:\alpha:1]\\ p_4=[1:\bar{\alpha}:1]\\ p_5=[1:1:\alpha]\\ p_6=[1:1:\bar{\alpha}]$} & $\frac{2\alpha+1}{270}(-\ell_1^6+\ell_2^6-\ell_3^6+\ell_4^6-\ell_5^6+\ell_6^6)$\\[55 pt]
		\hline
		$f_2$ & \parbox[t][][t]{2.3 cm}{$p_1=[1:\e:1]\\ p_2=[1:\be:1]\\ p_3=[0:\e:1]\\ p_4=[0:\be:1]\\ p_5=[1:\e:0]\\ p_6=[1:\be:0]$} & $\frac{2\e-1}{10}(-\ell_1^6+\ell_2^6+\ell_3^6-\ell_4^6+\ell_5^6-\ell_6^6)$\\[55 pt]
		\hline
		$f_3$ & \parbox[t][][t]{2.3 cm}{$p_1=[\om:1:\om]\\ p_2=[1:1:\om]\\ p_3=[\om:1:0]\\ p_4=[1:1:0]\\ p_5=[1:0:\om]\\ p_6=[1:0:1]$} & $\frac{2\om-1}{90}(\ell_1^6-\ell_2^6-\ell_3^6+\ell_4^6+\ell_5^6-\ell_6^6)$\\[55 pt]
		\hline
		$f_4$ & \parbox[t][][t]{2.3 cm}{$p_1=[\bom:1:\bom]\\ p_2=[1:1:\bom]\\ p_3=[\bom:1:0]\\ p_4=[1:1:0]\\ p_5=[1:0:\bom]\\ p_6=[1:0:1]$} & $\frac{2\bom-1}{90}(\ell_1^6-\ell_2^6-\ell_3^6+\ell_4^6+\ell_5^6-\ell_6^6)$\\[55 pt]
		\hline
		$f_5$ & \parbox[t][][t]{2.3 cm}{$p_1=[\beta:2:\bar{\beta}]\\ p_2=[\bar{\beta}:2:\beta]\\ p_3=[\beta:2:\beta]\\ p_4=[\bar{\beta}:2:\bar{\beta}]\\ p_5=[1:0:i]\\ p_6=[1:0:\bar{i}]$} & \raisebox{-5 pt}{$\dfrac{\ell_1^6+\ell_2^6-\ell_3^6-\ell_4^6-8i\ell_5^6-8i\ell_6^6}{1920}$}\\[55 pt]
	\end{tabular}
	
	\
	
	\caption{Waring decompositions of the forms $f_i$ in Theorem~\ref{thm:sixlines}.  The points $p_i$ give the coefficients of the linear forms $\ell_i$.}\label{tbl:Waring}
\end{table}

Now we explain the computations that led us to Theorem~\ref{thm:sixlines}.  We first reduce the number of variables needed.

\begin{lemma}
	If $\A$ is an irreducible arrangement in $\mathbb{P}^2$ then we can change variables so that $f=\Q(\A)$ has the form $f=xyz(x+y+z)\ell_1\ell_2\cdots\ell_t$, where $\ell_1,\ldots,\ell_t$ are linear forms.
\end{lemma}
\begin{proof}
	If $\A$ is irreducible then $f=\Q(\A)$ must have three factors which are linearly independent (otherwise $\A$ will decompose as a product of a one or two dimensional arrangement with the `empty' arrangement).  Furthermore $f$ must have at least four factors since otherwise it will decompose as a product of three one-dimensional arrangements.
	
	Changing variables, we may assume that $f$ has the form $f=xyz\ell_0\cdots\ell_t$ ($t\ge 0$).  We claim that $f$ has a collection of four factors no three of which are linearly dependent.  Suppose for a contradiction that every collection of four factors of $f$ has a subset of three factors which are linearly dependent.  Applying this supposition to the collection $\{x,y,z,\ell_i\}$ yields that one of the subsets $\{x,y,\ell_i\}$, $\{x,z,\ell_i\}$, or $\{y,z,\ell_i\}$ is linearly dependent.  Hence $\ell_i$ must be a linear form in only two variables for $i=0,\ldots,t$.  If each $\ell_i$ ($i=0,\ldots,t$) is a function of the \textit{same} two variables, the arrangement clearly decomposes as a product.  Hence we may assume without loss that $\ell_1=x+\alpha y$ and $\ell_2=x+\beta z$, where $\alpha,\beta\neq 0$.  But then $y,z,x+\alpha y,x+\beta z$ forms a collection of four factors of $f$ no three of which are linearly independent, proving the claim.
	
	Since $f$ has a collection of four factors no three of which are linearly independent, we can change variables to make three of these factors $x,y,$ and $z$.  The fourth factor must involve all three variables, hence we can apply scaling in the $x,y$ and $z$ directions to normalize the coefficients of the fourth factor to one.  Thus $f$ can be written in the form $f=xyz(x+y+z)\ell_1\cdots\ell_t$.
\end{proof}

\begin{corollary}\label{cor:varchange}
	If $(\A,\m)$ is an irreducible multi- arrangement in $\mathbb{P}^2$ with six lines, then there is a change of variables so that $\Q(\A,\m)=xyz(x+y+z)\ell_1\ell_2$, with $\ell_1$ and $\ell_2$ linear forms.
\end{corollary}


\begin{definition}\label{def:Catalecticant}
	Let $f\in S$ be a form of degree $d$ and $0\le t\le d$ an integer.  The map $\mbox{Cat}_f(t):R_t\to S_{d-t}$ defined by $\vp\to \vp\circ f$ is the \textit{catelecticant map}.  Choosing the usual basis of monomials for $R_t$ and $S_{d-t}$, we obtain the corresponding \textit{catalecticant matrix}.  Abusing notation, we will refer to this matrix also as $\mbox{Cat}_f(t)$.  The rows of $\mbox{Cat}_f(t)$ correspond to monomials in the basis of $S_{d-t}$, and the columns of $\mbox{Cat}_f(t)$ correspond to monomials in the basis of $R_t$.  Suppose $X^\alpha$ is a monomial in $R_d$ and $x^\beta$ is a monomial in $S_{d-t}$.  The entry of $\mbox{Cat}_f(t)$ in the row corresponding to $X^\alpha$ and column corresponding to $x^\beta$ is the coefficient of the monomial $x^\beta$ in $\frac{\partial{f}}{\partial x^\alpha}$. It is straightforward to see that $\ker(\mbox{Cat}_f(t))$ is $\text{Ann}_R(f)_t$.
\end{definition}

We return now to the computation at hand.  By Corollary~\ref{cor:varchange} we make a change of variables so that $f=xyz(x+y+z)\ell_1\ell_2$.  Introducing symbolic constants $a,b,c,d,e,$ and $f$ we can write $f=xyz(x+y+z)(ax+by+cz)(dx+ey+fz)$.  Now consider the condition in Theorem~\ref{thm:sixlines} that $\dim \text{Ann}_R(f)_3\ge 3$.  Here $R=\mbk[X,Y,Z]$ and $S=\mbk[x,y,z]$.  Using Definition~\ref{def:Catalecticant}, we see that $\text{Ann}_R(f)_3=\ker \mbox{Cat}_f(3):R_3\to S_3$.  Evidently $\mbox{Cat}_f(3)$ is a ten by ten matrix with entries of bi-degree $(1,1)$ in the variables $a,b,c$ and $d,e,f$; this matrix is shown in~\cite[Example~III.3.2]{W06}.  To say $\dim \text{Ann}_R(f)_3\ge 3$ is equivalent to imposing that $\mbox{rank}(\mbox{Cat}_f(3))\le 7$.  Thus the forms from Theorem~\ref{thm:sixlines} can be found as the zero locus of the seven by seven minors of this matrix.  As one may imagine, this approach is computationally infeasible.

To impose the rank condition we use an idea from~\cite{bates2009numerical} which reduces computation by introducing many auxiliary variables.  Explicitly, we introduce a ten by three matrix $B$ whose first three rows form a three by three identity matrix and whose remaining entries are filled with new variables:
\[
B=\begin{bmatrix}
1 & 0 & 0\\
0 & 1 & 0\\
0 & 0 & 1\\
A & B & C\\
D & E & F\\
G & H & V\\
J & K & L\\
M & N & O\\
P & Q & R\\
S & T & U\\
\end{bmatrix}.
\]
We then impose the condition $\mbox{Cat}_f(3)B=0$; this guarantees that $\mbox{Cat}_f(3)$ will have rank at most $7$.  This yields $30$ equations of total degree three in the $27$ variables $a,b,c,d,e,f,$ $A,\ldots,V$ (we replace the variable $I$ with $V$ since this is reserved for the imaginary unit in \textsc{Bertini}).  Since we only look for solutions up to constant multiple in the variable groups $a,b,c$ and $d,e,f$, we seek solutions in the $25$ dimensional space $\mathbb{P}^2\times\mathbb{P}^2\times \mathbb{C}^{21}$.  In \textsc{Bertini} we can specify this by using the option for homogeneous variable groups.  However we still must square the system, which we do by taking $25$ random linear combinations of the $30$ equations resulting from $\mbox{Cat}_f(3)B=0$.  The system of $25$ equations can now be solved by \textsc{Bertini} (on a personal laptop this computation is likely to take days -- we performed this computation on a local cluster).  We post-process the output by projecting onto the coordinates corresponding to $a,b,c,d,e,f$, removing duplicates, and then removing solutions that correspond to permuting the variables (a permutation fixes the first four factors of $f$ but acts non-trivially on the latter two factors).  This yields the list in Theorem~\ref{thm:sixlines}.  The scripts in \textsc{Macaulay2},~\textsc{Bertini}, and \textsc{Sage} which we used to find the forms in Theorem~\ref{thm:sixlines} and verify their properties may be found under the Research tab at~\url{https://midipasq.github.io/}.

\section{Conclusions and further questions}
There are two main results of this paper.  The first is a bound on the initial degree of the apolar ideal of a generic arrangement, attained using defining equations of star configurations from~\cite{GHG13}.  From this we obtained a necessary condition on the size of a generic arrangement with a complete intersection apolar algebra, as well as a lower bound on the Waring rank of a generic arrangement.  A subsequent question raised by Wakefield~\cite{W06} remains wide open -- is the apolar algebra of a generic arrangement ever a complete intersection?  To this we add two additional questions concerning the optimality of Proposition~\ref{prop:genericlowdeg} and Corollary~\ref{cor:genericWaring}.  First, are there arbitrarily large generic arrangements in $\mathbb{P}^n$ whose apolar ideals have initial degree $n+1$?  Second, are there arbitrarily large generic arrangements in $\mathbb{P}^n$ whose Waring rank is $\binom{2n}{n}$?

The second main result of this paper is the use of apolar algebras and numerical algebraic geometry to determine the irreducible multi-arrangements with six lines in $\mathbb P^2$ with minimal Waring rank.  We determined that, up to a change of coordinates, there are six irreducible multi-arrangements  that have cactus rank equal to six, five of which also have Waring rank equal to six.  These results are summarized  in  Theorem 5.1 (*). The (*) indicates that this is a ``numerically established theorem" and thus falls short of being a rigorously proved theorem. While one can check that each of these forms has the claimed Waring decomposition, one can't be certain that there do not exist further examples without further work. Thus, an obvious extension of this paper, that needs to be  carried  out, would be to either provide an alternate approach to establish that these are  the only such forms that have this property or else utilize software such as \textsc{alphaCertified} for \textsc{Bertini} to give a theoretical guarantee that the computations are correct. At present, the way we have chosen to make the computations is too expensive to carry out using  \textsc{alphaCertified} for \textsc{Bertini} on the system that we used. 

The general problem of determining the degree $d$ irreducible multi-arrangements in $\mathbb P^n$ that have minimal Waring rank (and minimal cactus rank) is currently out of  reach but  we leave it  as a suggestion for  a further path of research. It is worth  noting that each of the extremal  examples we found has interesting combinatorial properties.  In particular, after a change of coordinates, one is the  defining ideal of the $A_3$ braid arrangement. Another is half  of the Hessian arrangement. Perhaps there is a clue in the structure of these examples  that can help one search for higher degree extremal examples.  One promising avenue is to look for extremal behavior among the simplicial line arrangements catalogued by Grunbaum~\cite{Grun09}; such arrangements have recently led to interesting examples for the \textit{containment problem} between regular and symbolic powers~\cite{SM17}.
For now, we leave this as an open problem for the interested reader.

\section{Acknowledgements}
We thank Tanner Strunk for running our \textsc{Bertini} script on a local CSU computer cluster.  We would also like to thank Max Wakefield for pointing out to us that $f_1$ in Table \ref{tbl:6LinesS} is half of the Hessian arrangement.


\end{document}